\documentclass[graybox]{svmult}

\usepackage{amsmath,graphicx,amssymb,amsfonts}
\usepackage{mathptmx}       
\usepackage{helvet}         
\usepackage{courier}        
\usepackage{type1cm}        
\usepackage{makeidx}         
\usepackage{graphicx}        
\usepackage{multicol}        
\usepackage[bottom]{footmisc}

\makeindex             

\begin{document}\sloppy

\title*{Signed group orthogonal designs and their applications}
\author{Ebrahim Ghaderpour}
\institute{Ebrahim Ghaderpour\at Department of Earth and Space Science and Engineering, York University \\ 4700 Keele Street, Toronto, Ontario, Canada \ M3J 1P3 \\ \email{ebig2@yorku.ca}}
%
%
\maketitle

\abstract*{Craigen introduced and studied  {\it signed group Hadamard matrices}  extensively in \cite{Craigenthesis, Craigen}.  Livinskyi \cite{Ivan}, following Craigen's lead,  studied and provided a better estimate for the asymptotic existence of signed group Hadamard matrices and consequently  improved the asymptotic existence of Hadamard matrices. In this paper, we introduce and study signed group orthogonal designs. The main results include a method  for finding signed group orthogonal designs for any $k$-tuple of positive integer and then an application  to obtain orthogonal designs from signed group orthogonal designs, namely, for any $k$-tuple $\big(u_1, u_2, \ldots, u_{k}\big)$ of positive integers, we show that there is an integer $N=N(u_1, u_2, \ldots, u_k)$ such that for each $n\ge N$, a full orthogonal design (no zero entries) of type $\big(2^nu_1,2^nu_2,\ldots,2^nu_{k}\big)$ exists . This is an alternative approach to the results obtained in \cite{EK}.
\\ \\ \noindent {\bf Keywords} Asymptotic existence, Circulant matrix, Hadamard matrix, Orthogonal design, Signed group.}

\abstract{Craigen introduced and studied  {\it signed group Hadamard matrices}  extensively in \cite{Craigenthesis, Craigen}.  Livinskyi \cite{Ivan}, following Craigen's lead,  studied and provided a better estimate for the asymptotic existence of signed group Hadamard matrices and consequently  improved the asymptotic existence of Hadamard matrices. In this paper, we introduce and study signed group orthogonal designs. The main results include a method  for finding signed group orthogonal designs for any $k$-tuple of positive integer and then an application  to obtain orthogonal designs from signed group orthogonal designs, namely, for any $k$-tuple $\big(u_1, u_2, \ldots, u_{k}\big)$ of positive integers, we show that there is an integer $N=N(u_1, u_2, \ldots, u_k)$ such that for each $n\ge N$, a full orthogonal design (no zero entries) of type $\big(2^nu_1,2^nu_2,\ldots,2^nu_{k}\big)$ exists . This is an alternative approach to the results obtained in \cite{EK}.
\\ \\ \noindent {\bf Keywords} Asymptotic existence, Circulant matrix, Hadamard matrix, Orthogonal design, Signed group.}

\section{Introduction}
\label{sec:1}
A {\it signed group} $S$ (see \cite{Craigen}) is a group with a distinguished central element, an element that commutes with all elements of the group, of order two. Denote the unit of a group as 1 and the distinguished central element of order two as -1.  In every signed group, the set $\{1,-1\}$ is a normal subgroup, and we call the number of elements in the quotient group $S/\big<-1\big>$ the order of signed group $S.$ So, a signed group of order $n$ is a group of order $2n.$ 
A signed group $T$ is called a {\it signed subgroup} of a signed group $S,$ if $T$ is a subgroup of $S$ and the distinguished central elements of $S$ and $T$ coincide. We denote this relation by $T\le S.$ 

\begin{example} There are a number of signed groups with different applications. We present some of them used in this work:

$(i)$ \ \ The trivial signed group $S_{\mathbb R}=\{1,-1\}$ is a signed group of order $1$. 

$(ii)$ \ The complex signed group $S_{\mathbb C}=\big<i; \ i^2=-1\big>=\{\pm 1, \pm i\}$ is a signed group of order $2$. 

$(iii)$ The Quaternion signed group $S_Q=\big<j,k; \ j^2=k^2=-1, jk=-kj\big>=\big\{\pm 1, \pm j, \pm k, \pm jk\big\}$ is a signed group of order $4$.

$(iv)$ The set of all monomial $\{0, \pm 1\}$-matrices of order $n$, $SP_n$, forms a group of order $2^nn!$ and a signed group of order $2^{n-1}n!.$
\end{example}

Let $S$ and $T$ be two signed groups. A {\it signed group homomorphism} $\phi: S \rightarrow T$ is a map such that for all $a,b \in S$, $\phi(ab)=\phi(a)\phi(b)$ and $\phi(-1)=-1.$
A {\it remrep}  ({\it real monomial representation}) is a signed group homomorphism  $\pi: S\rightarrow SP_n$.  A {\it faithful remrep} is a one to one remrep. 

Let $R$ be a ring with unit $1_R,$ and let $S$ be a signed group with distinguished central element $-1_S.$  Then  
$R[S]:=\big\{\sum_{i=1}^n r_is_i;  \ r_i\in R, \ s_i\in P\big\}$ is the signed group ring, where $P$ is a set of coset representatives of $S$ modulus $\big<-1_S\big>$ and for $r\in R, \ s\in P$, we make the identification $-rs=r(-s)$. Addition is defined termwise, and multiplication is defined by linear extension. For instance,  $r_1s_1(r_2s_2+r_3s_3)=r_1r_2s_1s_2+r_1r_3s_1s_3$, where $r_i\in R$ and $s_i\in P$, $i\in \{1,2,3\}$.

In this work, we choose $R=\mathbb{R}$. 
Suppose $x\in\mathbb{R}[S].$ Then $x=\sum_{i=1}^n r_is_i$, where $ r_i\in \mathbb{R}, \ s_i\in P$. The  {\it conjugation} of $x$, denoted $\overline{x}$, is defined as  
$\overline{x}:=\sum_{i=1}^n r_is^{-1}_i$.
Clearly, the conjugation is an involution, i.e., $\overline{\overline x}=x$ for all $x\in \mathbb{R}[S],$ and $\overline{xy}=\bar{y}\bar{x}$ for all $x,y\in \mathbb{R}[S].$ As an example, $\overline{\sqrt{2}j+3jk}=\sqrt{2}j^{-1}+3(jk)^{-1}=-\sqrt{2}j-3jk,$ where $j,k\in S_Q.$

For an $m\times n$ matrix $A=[a_{ij}]$ with entries in $\mathbb{R}[S]$ define its adjoint as an $n\times m$ matrix $A^*=\overline{A}^t=[\overline{a}_{ji}].$ Let $S$ be a signed group, and let $A=[a_{ij}]$ be a square matrix such that $a_{ij}\in\big\{0, \epsilon_1x_1, \ldots, \epsilon_kx_k\big\},$  where $\epsilon_{\ell}\in S$ and $x_{\ell}$ is a variable, $1\le \ell \le k.$ For each $a_{ij}=\epsilon_{\ell}x_{\ell}$ or 0, let $\overline{a}_{ij}=\overline{\epsilon}_{\ell}x_{\ell}$ or 0, and $|a_{ij}|=|\epsilon_{\ell}x_{\ell}|=x_{\ell}$ or 0. We define ${\rm abs}(A):=\big[|a_{ij}|\big].$ We call $A$  {\it quasisymmetric}, if ${\rm abs}(A)={\rm abs}(A^*),$ 
where $A^*=[\overline{a}_{ji}].$ Also, $A$ is called {\it  normal} if $AA^*=A^*A.$ 
The {\it support} of $A$ (see \cite{Craigen}) is defined by ${\rm supp}(A):=\big\{{\rm positions \ of \ all \ nonzero \ entries \ of \ } A\big\}$. 

Suppose $A=\big(a_1, a_2, \ldots, a_n\big)$ and $B=\big(b_1, b_2, \ldots, b_n\big)$ are two sequences with elements from $\big\{0,\epsilon_1x_1, \ldots, \epsilon_kx_k\big\},$ where the $x_k$'s are variables and $\epsilon_k\in S$  $(1\le k\le n)$ for some signed group $S$. We use  $A_{\overline{R}}$ to denote the sequence whose elements are those of $A$, conjugated and in reverse order (see \cite{CHK}), i.e.,
$A_{\overline{R}}=\big(\overline{a}_n, \ldots, \overline{a}_2, \overline{a}_1\big).$
We say $A$ is {\it quasireverse} to $B$ if ${\rm abs}(A_{\overline{R}})={\rm abs}(B).$ 

A circulant matrix $C={\rm circ}\big(a_1, a_2, \ldots, a_n\big)$ (see \cite[chap. 4]{GS}) can be written as $C=a_1I_n+\sum_{k=1}^{n-1}a_{k+1}U^k,$ where $U={\rm circ}\big(0,1,0,\ldots,0\big).$ Therefore, any two circulant matrices of order $n$ with commuting entries commute. 
If $C={\rm circ}\big(a_1, a_2, \ldots, a_n\big),$ then $C^*={\rm circ}\big(\overline{a}_1, \overline{a}_n, \ldots, \overline{a}_2\big)$.

We use the notation $u_{(k)}$  to show $u$ repeats $k$ times. Suppose that $A$ and $B$ are two sequences of length $n$ such that $A$ is quasireverse to $B$. Let $D={\rm circ}\big(0_{(a+1)},A, 0_{(2b+1)}, B, 0_{(a)}\big)$, where $a$ and $b$ are nonnegative integers and let $m=2a+2b+2n+2$. Then $D^*={\rm circ}\big(0_{(a+1)},B_{\overline{R}}, 0_{(2b+1)}, A_{\overline{R}}, 0_{(a)}\big)$ and ${\rm abs}(D)={\rm abs}(D^*).$ Hence, $D$ is a quasisymmetric circulant matrix of order $m$.

The {\it non-periodic autocorrelation function} \cite{KHH}  of a sequence $A=(x_1, \ldots, x_n)$ of commuting square complex matrices of order $m,$ is defined by
$$N_{A}(j):=\left\{
      \begin{array}{l l}
      \displaystyle \sum_{i=1}^{n-j} x_{i+j}x_i^* \ \ \ {\rm if} \ \  j=0,1,2,\ldots, n-1& \\
       0 \ \ \ \ \ \ \ \ \ \ \ \ \ \ j\ge n& 
       \end{array} \right.$$
where $x_i^*$ denotes the conjugate transpose of $x_i.$
A set $\{A_1, A_2, \ldots, A_{\ell}\}$ of sequences (not necessarily in the same length) is said to have zero autocorrelation if for all $j>0,$
$\sum_{k=1}^{\ell}N_{A_k}(j)=0.$  Sequences having zero autocorrelation are called {\it complementary}.

A pair $(A;B)$ of $\{\pm 1\}$-complementary sequences of length $n$ is called a {\it Golay pair}  of length $n$, and a pair $(A_1;B_1)$  of $\{\pm x, \pm y\}$-complementary sequences  of length $n_1$ is called a {\it Golay pair in two variables $x$ and $y$} of length $n_1$. The length $n$ is called  {\it Golay number}. 
Similarly, a pair $(C;D)$ of $\{\pm 1, \pm i\}$-complementary sequences  of length $m$ is called a {\it complex Golay pair} of length $m$, and a pair $(C_1;D_1)$  of $\{\pm x, \pm i x, \pm y, \pm iy\}$-complementary sequences of length $m_1$ is called a {\it complex Golay pair in two variables $x$ and $y$} of length $m_1$. The length $m$ is called  {\it complex Golay number}. In this paper, the sequences $A_1$ and $C_1$ are assumed to be quasireverse to $B_1$ and $D_1$, respectively. 

Craigen, Holzmann and Kharaghani in \cite{CHK} showed that if $g_1$ and $g_2$ are complex Golay numbers and $g$ is an even Golay number, then $gg_1g_2$ is a complex Golay number. Using this, they showed the following theorem.
\begin{theorem}\label{CGN2}
All numbers of the form $m=2^{a+u}3^b5^c11^d13^e$ are complex Golay numbers, where $a,b,c,d,e$ and $u$ are non-negative integers such that $b+c+d+e\le a+2u+1$ and $u\le c+e.$
\end{theorem}

The following lemma is immediate from the definition of complex Golay pair.
\begin{lemma}\label{twovariables}
Suppose that $(A;B)$ is a complex Golay pair of length $m$. Then $\big(\!(xA,yB);(yA,-xB)\!\big)$ is a complex Golay pair of length $2m$ in two variables $x$ and $y$. 
\end{lemma}

From Theorem \ref{CGN2} and Lemma \ref{twovariables}, we have the following result.
\begin{corollary}\label{CGN1}
There is a complex Golay pair in two variables of length $n=2^{a+u+1}3^b5^c11^d13^e$, where $a,b,c,d,e$ and $u$ are non-negative integers such that $b+c+d+e\le a+2u+1$ and $u\le c+e.$
\end{corollary}

In Section \ref{sec2}, we introduce signed group orthogonal designs, and will show some of their properties. Then as one of their applications, in Theorem \ref{sodtood}, we show how to obtain orthogonal designs from signed group orthogonal designs. In Section \ref{sec4}, using signed group orthogonal designs, we prove Theorems \ref{thirdbound} and \ref{lastbound} that give two different bounds for the asymptotic existence of orthogonal designs, namely, for any $k$-tuple $(u_1, u_2, \ldots, u_{k})$ of positive integers, there is an integer $N=N(u_1, u_2, \ldots, u_k)$ such that a full orthogonal design of type $\big(2^nu_1,2^nu_2,\ldots,2^nu_{k}\big)$ exists 
for each $n\ge N$.

In this paper, $P:= \left[ \begin {array}{cc} 0&1\\ 1&0\end {array} \right]\!\!,$
$Q:= \left[ \begin {array}{cc} 1&0\\ 0&-\end {array} \right]\!\!,$ 
$R:= \left[ \begin {array}{cc} 0&1\\ -&0\end {array} \right]$  and $I_d$ is the identity matrix of order $d$, where $-$ is $-1$.
\section{Signed group orthogonal designs and some of their properties}\label{sec2}

A {\it signed group orthogonal design}, SOD,  of type $\big(u_1, \ldots, u_k\big),$ where $u_1, \ldots, u_k$ are positive integers, and of order $n,$ is a square matrix $X$ of order $n$ with entries
from $\{0, \epsilon_1 x_1, \ldots, \epsilon_k x_k\},$ where the $x_i$'s are variables and
$\epsilon_j\in S,$ $1\le j\le k,$ for some signed group $S$,  that satisfies
\begin{align*}
XX^{*}= \Bigg(\sum_{i=1}^k u_i x_i^2 \Bigg)I_n. 
\end{align*}
We denote it by $SOD\big(n; \ u_1, \ldots, u_k\big).$

Equating all variables to 1 in any SOD of order $n$ results in a {\it signed group weighing matrix} of order $n$ and weight $w$ which is denoted by $SW(n,w)$, where $w$ is the number of nonzero entries in each row (column) of the SOD.  We call an SOD with no zero entries a full SOD.
Equating all variables to 1 in any full SOD of order $n$ results in a {\it  signed group Hadamard matrix} of order $n$ which is denoted by $SH(n,S).$

Craigen $\cite{Craigen}$ proved the following fundamental theorem and applied it to demonstrate a novel and new method for the asymptotic existence of signed group Hadamard matrices and consequently Hadamard matrices.
\begin{theorem}
For any odd positive integer $p$, there exists a circulant $SH\big(2p, SP_{2^{2N(p)-1}}\big).$
\end{theorem}
\begin{remark}
An SOD over the Quaternion signed group $S_Q$ is called a {\it Quaternion orthogonal design}, QOD. 
An SOD over the complex signed group $S_{\mathbb{C}}$ is called a {\it complex orthogonal design}, COD. 
An SOD over the trivial signed group $S_{\mathbb{R}}$ is called an {\it orthogonal design}, OD. 
\end{remark}
\begin{lemma}\label{SWnormal}
Every $SW(n, w)$ over a finite signed group is normal. 
\end{lemma}
\begin{proof}\smartqed
Suppose that $WW^*=wI_n,$ where the entries in $W$ belong to a signed group $S$ of order $m.$ We show that $WW^*=W^*W$. The space of
all square matrices of order $n$ with entries in $\mathbb{R}[S]$ has the standard basis with $mn^2$ elements over the field $\mathbb{R}$. Thus, there exists an integer $u$ such that $$c_1W+c_2W^2+\cdots+ c_uW^u=0,$$ where $c_u\neq 0$, and $c_i\in \mathbb{R} \ (1\le i\le u).$ Multiplying the above equality from the right by $(W^*)^{u-1},$ $$c_1w(W^*)^{u-2}+c_2w^2(W^*)^{u-3}+\cdots +c_uw^{u-1}W=0.$$  Hence $W$ is a polynomial in $W^*$, and so $WW^*=W^*W.$ \qed
\end{proof}
\begin{theorem}\label{SOD}
 A necessary and sufficient condition that there is a $SOD\big(n; \ u_1, \ldots, u_k\big)$ over a signed group $S,$ is that there exists a family 
$\{A_1, \ldots, A_k\}$ of pairwise disjoint square matrices of order $n$ with  entries from $\{0, S\}$ satisfying
\begin{align}
&A_iA_i^*=u_iI_n, \ \ \ \ \ \ \ \ \ \ \ \ \ \ \ \ \ 1\leq i \leq k, \label{necessary1} \\
&A_iA_j^*=-A_jA_i^*, \ \ \ \ \ \ \ \ \ \ \ \ 1\leq i\neq j \leq k. \label{necessary2}
\end{align}
\end{theorem}
\begin{proof}\smartqed
 Suppose that there is a $A=SOD\big(n; \ u_1, \ldots, u_k\big)$ over a signed group $S.$ One can write 
\begin{eqnarray} \label{x}
 A=\sum_{m=1}^kx_mA_m, 
\end{eqnarray}
 where the $A_i$'s are
square matrices of order $n$ with  entries from $\{0, S\}.$ Since the entries in $A$ are linear monomials in the $x_i,$ the $A_i$'s are disjoint. Since $A$ is an SOD,
\begin{align}\label{AA}
\displaystyle AA^*=\Bigg(\sum_{i=1}^ku_ix_i^2\Bigg)I_n,
\end{align}
and so by using
 \eqref{x},  
\begin{align}\label{Aj}
\sum_{m=1}^{k}x_m^2A_mA_m^*+\sum_{i=1}^k\sum_{j=i+1}^k x_ix_j(A_iA_j^*+A_jA_i^*)=\Bigg(\sum_{i=1}^ku_ix_i^2\Bigg)I_n.
\end{align}
In the above equality, for each $1\leq i \leq k,$ let $x_i=1$ and $x_{j}=0$ for all $1\le j\le k$ and $j\neq i,$  to get  \eqref{necessary1} and therefore  \eqref{necessary2}. 

On the other hand, if $\{A_1, \ldots, A_k\}$ are pairwise disjoint square matrices of order $n$ with  entries from $\{0, S\}$ which satisfy  \eqref{necessary1} and  \eqref{necessary2}, then the left hand side of the equality  \eqref{Aj} gives us  \eqref{AA}.\qed
\end{proof}
\begin{remark}
Equation \eqref{AA} implies equations \eqref{necessary1} and \eqref{necessary2}. 
Multiply  \eqref{necessary2} from the left by $A_i^*$ and then from the right by $A_i$ to get $A_j^*A_i=-A_i^*A_j$ for $1\leq i\neq j \leq k.$ Therefore,
by Lemma \ref{SWnormal},
\begin{align*}
A^*A=\sum_{m=1}^{k}x_m^2A_m^*A_m+\sum_{i=1}^k\sum_{j=i+1}^k x_ix_j(A_i^*A_j+A_j^*A_i)=\Bigg(\sum_{i=1}^ku_ix_i^2\Bigg)I_n.
\end{align*}
Thus, $AA^*=A^*A.$
 It means that every SOD over a finite signed group is normal.
\end{remark}
\begin{lemma}\label{musteven}
There does not exist any full SOD of order $n>1,$ if $n$ is odd. 
\end{lemma}
\begin{proof}\smartqed
Assume that there is a full SOD of order $n>1$ over a signed group $S$.  Equating all variables to 1 in the SOD, one obtains a $SH(n,S)=[h_{ij}]_{i,j=1}^n.$ One may multiply each column of the $SH(n,S),$ from the right, by the inverse of corresponding entry of its first row, $\overline{h}_{1j},$ to get an equivalent $SH(n,S)$ with the first row all 1 (see  \cite{Craigen, CHK} for the definition of equivalence). By orthogonality of the rows of the $SH(n,S)$, the number of occurrences of a given element $s\in S$ in each subsequent row must be equal to the number of occurrences of $-s$. Therefore, $n$ has to be even.\qed
\end{proof}

\section{Some applications of signed group orthogonal designs }
In this section, we adapt the methods of Livinskyi \cite{Ivan} to obtain generalizations and improvements of his results about Hadamard matrices in the much more general setting of ODs.

Suppose that we have a remrep $\pi: S\rightarrow SP_m.$ We extend this remrep to a ring homomorphism $\pi: \mathbb{R}[S]\rightarrow M_m[\mathbb{R}]$ linearly by $\pi\big(r_1s_1+\cdots + r_ns_n\big)=r_1\pi\big(s_1\big)+\cdots+r_n\pi\big(s_n\big).$ Since for every matrix $A\in SP_m$ we have $A^{-1}=A^t$, for every $s\in S$, $\pi\big(\overline{s}\big)=\pi(s)^{-1}=\pi(s)^t.$

Next theorem shows how one can obtain ODs from SODs.
\begin{theorem}\label{sodtood}
Suppose that there exists a $SOD\big(n; \ u_1, \ldots, u_k\big)$ for some signed group $S$ equipped with a remrep $\pi$ of degree $m,$ where $m$ is the order of a Hadamard matrix. Then there is an $OD\big(mn; \ mu_1, \dots, mu_k\big).$  
\end{theorem}
\begin{proof}\smartqed
Suppose that there exists a $SOD\big(n; \ u_1, \ldots, u_k\big)$ for some signed group $S.$ By Theorem \ref{SOD}, there are pairwise disjoint matrices $A_1, \ldots, A_k$ of order $n$ with entries in $\{0, S\}$ such that 
\begin{align}
& \ A_{\alpha} A_{\alpha}^*= u_{\alpha}I_n, \ \ \ \ \ \ \ \ \ \ \ \ 1\le \alpha\le k, \label {con1}\\
& \  A_{\alpha} A_{\beta}^*=-A_{\beta} A_{\alpha}^*, \ \ \ \ \ \ \ 1\le \alpha\neq \beta \le k. \label {con2}
\end{align}
\noindent Let $\pi: \ S\rightarrow SP_m$ be a remrep of degree $m,$ and $H$ be a Hadamard matrix of degree $m.$ Also, for each $1\le \alpha\le k,$ let
$$B_{\alpha}=\Big[\pi\big(A_{\alpha}[i,j]\big)H\Big]_{i.j=1}^n.$$  By Proposition $1.1$ in \cite{GS}, it is sufficient to show that $B_{\alpha}$'s are pairwise disjoint matrices of order $mn,$ with $\{0,\pm 1\}$ entries such that 
\begin{align}
&B_{\alpha} B_{\alpha}^t= mu_{\alpha}I_{mn}, \ \ \ \ \ \ \ \ \ \ \ \ 1\le \alpha\le k,  \label {con21} \\
&B_{\alpha} B_{\beta}^t=-B_{\beta} B_{\alpha}^t, \ \ \ \ \ \ \ \ \ \ \ \ 1\le \alpha\neq \beta \le k. \label {con22}
\end{align}
Since $A_{\alpha}$'s are pairwise disjoint, so are $B_{\alpha}$'s (see \cite[chap. 1]{GS} for Hurwitz-Radon matrices and their properties). Let $1\le \alpha\neq \beta \le k$ and $1\le i,j \le n.$ Then
\begin{align}
\big(B_{\alpha} B_{\beta}^t\big)[i,j]&=\sum_{k=1}^n\pi\big(A_{\alpha}[i,k]\big)HH^t\pi\big(A_{\beta}[j,k]\big)^t\nonumber \\ &=m\sum_{k=1}^n\pi\big(A_{\alpha}[i,k]\big)\pi\big(\overline{A}_{\beta}[j,k]\big)\nonumber\\ &=m \pi\Big(\sum_{k=1}^n A_{\alpha}[i,k]\overline{A}_{\beta}[j,k]\Big) \nonumber\\ &=m\pi\Big(\big(A_{\alpha}A_{\beta}^*\big)[i,j]\Big)\label{left1}\\ &=m\pi\Big(\big(-A_{\beta}A_{\alpha}^*\big)[i,j]\Big) \ \ \ \ \ {\rm from \ \   \eqref{con2}}\nonumber \\ &=-m\pi\Big(\big(A_{\beta}A_{\alpha}^*\big)[i,j]\Big) \label {left2}
\end{align}
On the other hand, similarly, 
\begin{align}
\big(B_{\beta} B_{\alpha}^t\big)[i,j]&=\sum_{k=1}^n\pi\big(A_{\beta}[i,k]\big)HH^t\pi\big(A_{\alpha}[j,k]\big)^t\nonumber \\ &=m\sum_{k=1}^n\pi\big(A_{\beta}[i,k]\big)\pi\big(\overline{A}_{\alpha}[j,k]\big)\nonumber \\ &=m \pi\Big(\sum_{k=1}^n A_{\beta}[i,k]\overline{A}_{\alpha}[j,k]\Big) \nonumber \\ &=m\pi\Big(\big(A_{\beta}A_{\alpha}^*\big)[i,j]\Big) \label {right1}.
\end{align}
Comparing  \eqref{left2} and \eqref{right1}, one obtains  \eqref{con22}. If  $\alpha=\beta$ in  \eqref{left1}, then for $1\le i,j \le n,$
\begin{align*}
\big(B_{\alpha} B_{\alpha}^t\big)[i,j]&=m\pi\Big(\big(A_{\alpha}A_{\alpha}^*\big)[i,j]\Big) \\ &= m\pi\big(\gamma_{ij}u_{\alpha}\cdot 1_S\big)\ \ \ \ \ {\rm from \ \   \eqref{con1}}\\ &=m\gamma_{ij}u_{\alpha}I_m,
\end{align*}
where $\gamma_{ij}=1$ if $i=j,$ and 0 otherwise. Whence  \eqref{con21} follows. \qed
\end{proof}

In the following two corollaries, it is shown how to obtain ODs from CODs and QODs.
\begin{corollary}
If there exists a $COD\big(n; \ u_1, \ldots, u_k\big),$ then there exists an $OD\big(2n; \ 2u_1, \ldots, 2u_k\big)$.
\end{corollary}
\begin{proof}\smartqed
A $COD\big(n; \ u_1, \ldots, u_k\big)$ can be viewed as a $SOD\big(n; \ u_1, \ldots, u_k\big)$ over the complex signed group $S_{\mathbb{C}}.$ It can be seen that $\pi: S_{\mathbb C}\rightarrow SP_2$ defined by $$i \longrightarrow R=\left[ \begin {array}{cc} 0&1\\ -&0\end {array} \right]$$ is a remrep of degree 2, and so by Theorem \ref{sodtood}, there exists an $OD\big(2n; \ 2u_1, \ldots, 2u_k\big).$ \qed
\end{proof}
\begin{corollary}
If there exists a $QOD\big(n; \ u_1, \ldots, u_k\big),$ then there exists an $OD\big(4n; \ 4u_1, \ldots, 4u_k\big).$
\end{corollary}
\begin{proof}\smartqed
A $QOD\big(n; \ u_1, \ldots, u_k\big)$ can be viewed as a $SOD\big(n; \ u_1, \ldots, u_k\big)$ over the Quaternion signed group $S_Q.$ It can be seen that $\pi: S_Q\rightarrow SP_4$ defined by $$j \longrightarrow R\otimes I_2=\left[ \begin {array}{cccc} 0&0&1&0\\ 0&0&0&1 \\ -&0&0&0 \\ 0&-&0&0\end {array} \right] \ \ \ \ {\rm and}\ \ \ \ k\longrightarrow P\otimes R=\left[ \begin {array}{cccc} 0&0&0&1\\ 0&0&-&0 \\ 0&1&0&0 \\ -&0&0&0\end {array} \right]\!,$$
is a remrep of degree 4, and so by Theorem \ref{sodtood}, there exists an $OD\big(4n; \ 4u_1, \ldots, 4u_k\big).$\qed
\end{proof}

Following similar techniques in $\cite{Craigen, CHK123, Ivan}$, we have the following Lemma.
\begin{lemma}\label{ZS}
Suppose that $A$ and $B$ are two disjoint  circulant matrices of order $d$ with entries from
$\big\{0, \epsilon_1x_1, \ldots, \epsilon_kx_k\big\},$ where the $x_{\ell}$'s are variables, $\epsilon_{\ell}\in S \ (1\le \ell \le k)$ for $A$ and $\epsilon_{\ell}\in Z(S)$, the center of $S$, $(1\le \ell \le k)$ for $B.$ Also, assume $A$ is normal. If $$C= \left[ \begin {array}{cc} A+B&A-B\\ A^*-B^*&-A^*-B^*\end {array} \right]\!,$$ then $CC^*=C^*C=2I_2\otimes (AA^*+BB^*).$ Moreover, if $A$ and $B$ are both quasisymmetric and $S$ has a faithful remrep of degree $m$, then there exists a  circulant quasisymmetric normal matrix $D$ of order $d$ with entries from $\big\{0, \epsilon_1' x_1, \ldots, \epsilon_k' x_k\big\}$ and the same support as $A+B$ such that 
$DD^*=AA^*+BB^*,$ where $\epsilon'_{\ell}\in S' \ (1\le \ell \le k),$ and $S' \ge S$ is a signed group having a faithful remrep of degree $2m$.  
\end{lemma}
\begin{proof}\smartqed
It may be verified directly that $CC^*=C^*C=2I_2\otimes (AA^*+BB^*).$ To find matrix $D,$ first reorder the rows and columns of $C$ to get matrix $D_0$ which is a partitioned matrix of order $2d$ into $2\times 2$ blocks whose entries are the $(i,j)$, $(i+d,j)$, $(i,j+d)$ and $(i+d,j+d)$ entries of $C, \ 1\le i,j\le d.$ Applying the same reordering to
$2I_2\otimes (AA^*+BB^*),$ one obtains  $(AA^*+BB^*)\otimes 2I_2.$ Since $A$ and $B$ are disjoint and quasisymmetric, each non-zero block of $D_0$ will have one of the following forms 
$$\left[ \begin {array}{rr} \epsilon_ix_i&\epsilon_ix_i \\ \epsilon_jx_i&-\epsilon_jx_i\end {array} \right] \ \ \ \ \ \ {\rm or} \ \ \ \ \ \ 
\left[ \begin {array}{rr} \epsilon_ix_i&-\epsilon_ix_i \\ \epsilon_jx_i&\epsilon_jx_i\end {array} \right]\! ,$$ where $\epsilon_{\ell}\in S.$ Multiplying $D_0$ on the right by $\frac{1}{2}I_d\otimes \!\left[ \begin {array}{cc} 1&1 \\ 1&-\end {array} \right]$ yields a matrix $D_1$ of order $2d$ with entries from $\big\{0, \epsilon_1 x_1, \ldots, \epsilon_k x_k\big\}$ whose non-zero $2\times 2$ blocks have one of the forms $A_ix_i$ or $B_ix_i$, where
\begin{align}\label{D1}
A_i=\left[ \begin {array}{cc} \epsilon_i&0 \\ 0&\epsilon_j\end {array} \right] \ \ \ \ \ \ {\rm or} \ \ \ \ \ \ 
B_i=\left[ \begin {array}{cc} 0&\epsilon_i \\ \epsilon_j&0\end {array} \right]\! ,
\end{align} and such that
 $D_1D_1^*=D_1^*D_1= (AA^*+BB^*)\otimes I_2.$
The $A_i$'s and $B_i$'s in  \eqref{D1} form another signed group, $S'$. Now matrices of the form $$\epsilon\otimes I_2=\left[ \begin {array}{cc} \epsilon&0 \\ 0&\epsilon\end {array} \right]\! \! , \ \ \ \epsilon\in S,$$ form a signed subgroup of $S'$ which is isomorphic to $S.$ Therefore, one can identify this signed subgroup with $S$ itself and consider $S'$ as an extension of $S.$ Replacing every $2\times 2$ block of $D_1$ which is one the forms in  \eqref{D1} or zero with corresponding $\epsilon'_ix_i,  \ \epsilon'_i\in S'$ or zero, gives the required matrix $D.$ Note that we identify $\epsilon\otimes I_2\in S'$ with $\epsilon \in S.$ 

Now if $\pi: S\rightarrow SP_m'\le SP_m$ is a faithful remrep of degree $m$, then it can be verified directly that the map $\pi': S'\rightarrow SP_{2m}'\le SP_{2m}$ which is uniquely defined by 
$$\left[ \begin {array}{cc} \epsilon_i&0 \\ 0&\epsilon_j\end {array} \right]\rightarrow \left[ \begin {array}{cc} \pi(\epsilon_i)&0_m \\ 0_m&\pi(\epsilon_j)\end {array} \right]\! , \ \ \ \ \left[ \begin {array}{cc} 0&\epsilon_i \\ \epsilon_j&0\end {array} \right]\rightarrow \left[ \begin {array}{cc} 0_m&\pi(\epsilon_i) \\ \pi(\epsilon_j)&0_m\end {array} \right]\! ,$$
is a faithful remrep of degree $2m$, where $0_m$ denotes the zero matrix of order $m.$

Finally, since $A$ and $B$ are circulant, $C$ consists of four circulant blocks, so $D_0$ and $D_1$ are  block-circulant 
with block size $2\times 2;$ whence $D$ is circulant and quasisymmetric. \qed
\end{proof}

We now use Lemma \ref{ZS} and follow similar techniques  in $\cite{Craigen, Ivan}$ to show the following Theorem.
 \begin{theorem}\label{induct}
Suppose that $B_1, \ldots, B_n$ are disjoint quasisymmetric  circulant matrices of order $d$ with entries from $\big\{0, \epsilon_1x_1, \ldots, \epsilon_kx_k\big\},$ where  $\epsilon_{\ell}\in S_{\mathbb C},$ and the $x_{\ell}$'s are variables $(1\le \ell \le k)$, such that 
$$B_1B_1^*+\cdots +B_nB_n^*=\bigg(\sum_{\ell=1}^k u_{\ell}x_{\ell}^2\bigg)I_{d},$$
where the $u_{\ell}$'s are positive integers.
Then there is a quasisymmetric circulant $SOD\big(d; u_1, \ldots, u_k\big)$ for a signed group $S$ that admits a faithful remrep of degree $2^{n}$.
\end{theorem}
\begin{proof}\smartqed
$S_{\mathbb C}$ has a faithful remrep $\pi:S_{\mathbb C}\rightarrow SP'_2\le SP_2$ of degree 2 uniquely determined by $\pi(i)=R$, where $SP'_2=\big<R; \ R^2=-I\big>$. Applying Lemma \ref{ZS} to matrices $B_1$ and $B_2,$ one obtains a quasisymmetric normal circulant matrix $A_1$ of order $d$ with entries from $\big\{0, \epsilon_1^{(1)} x_1, \ldots, \epsilon_k^{(1)} x_k\big\},$ where  $\epsilon_{\ell}^{(1)}\in S_1$ $(1\le \ell \le k)$ such that $S_1\ge S_{\mathbb C}$ is a signed group with a faithful remrep of degree $2^2$.  
Also, $A_1A_1^*=B_1B_1^*+B_2B_2^*.$  Since ${\rm supp}(A_1)$ is the union of ${\rm supp}(B_1)$ and ${\rm supp}(B_2),$ $A_1$ is disjoint from $B_3, \ldots, B_n.$ 

Suppose that one has constructed a circulant quasisymmetric normal matrix $A_{r}$ of order $d$ with entries from $\big\{0, \epsilon_1^{(r)} x_1, \ldots, \epsilon_k^{(r)} x_k\big\},$ where $\epsilon_{\ell}^{(r)}\in S_{r} \ (1\le \ell \le k)$ such that $S_{r}\ge S_{r-1}$ is a signed group with a faithful remrep $\pi_r: S_r \rightarrow SP'_{2^{r+1}}\le SP_{2^{r+1}}$ of degree $2^{r+1}$. Moreover, $A_{r}$ is disjoint from $B_{r+2}, \ldots, B_n$ and
$$A_rA_r^*=B_1B_1^*+\cdots +B_{r+1}B_{r+1}^*.$$ 

By the assumption, $B_{r+2}$ is a quasisymmetric normal circulant matrix with entries from $\big\{0, \epsilon_1x_1, \ldots, \epsilon_kx_k\big\},$ where $\epsilon_{\ell}\in S_{\mathbb C}\ (1\le \ell \le k).$ One can view the $\epsilon_{\ell}$'s as elements in $Z\big(S_{r}\big)$ because we identified these elements as blocks $\pm I_{2^{r}}\otimes R$ and $\pm I_{2^{r+1}}$ in the proof of Lemma \ref{ZS} which commute with $\pi_r\big(\epsilon_{\ell}^{(r)}\big), \ 1\le \ell \le k.$ Therefore, by Lemma \ref{ZS}, there is a quasisymmetric normal circulant matrix $A_{r+1}$ with entries from $\big\{0, \epsilon_1^{(r+1)} x_1, \ldots, \epsilon_k^{(r+1)} x_k\big\},$ where $\epsilon_{\ell}^{(r+1)}\in S_{r+1} \ (1\le \ell \le k)$ such that $S_{r+1}\ge S_r$ is a signed group with a faithful remrep of degree $2^{r+2}$. Also, 
$$A_{r+1}A_{r+1}^*=A_rA_r^*+B_{r+2}B_{r+2}^*=B_1B_1^*+\cdots +B_{r+1}B_{r+1}^*+B_{r+2}B_{r+2}^*,$$ and by the same argument $A_{r+1}$ is disjoint from $B_{r+3}, \ldots, B_n.$ Applying this procedure $n-2$ times, there is a  quasisymmetric normal circulant matrix $A_{n-1}$ of order $d$ such that $$A_{n-1}A_{n-1}^*=B_1B_1^*+\cdots +B_nB_n^*=\bigg(\sum_{\ell=1}^k u_{\ell}x_{\ell}^2\bigg)I_{d},$$ which is a circulant quasisymmetric $SOD\big(d; \ u_1, \ldots, u_k\big)$ with the signed group \\ $S=S_{n-1}\ge S_{n-2} \ge \cdots \ge S_{\mathbb C}$ that admits a faithful remrep of degree $2^{n}$.\qed
\end{proof}
\begin{remark}
The circulant matrices in Theorem \ref{induct} are taken on the abelian signed group $S_\mathbb{C};$ however, if the signed group is not abelian, the circulant matrices that obtain from Lemma \ref{ZS} do not necessarily commute, and Theorem \ref{induct} may fail. As an example, if 
$B_1={\rm circ}(j,0)$ and $B_2={\rm circ}(0,k),$ where $j,k\in S_Q,$ then since $jk=-kj,$  $B_1B_2\neq B_2B_1.$ Therefore, Lemma \ref{ZS} does not apply in this case.
\end{remark}
 \begin{theorem}\label{cortoreal}
Suppose that $B_1, \ldots, B_n$ are disjoint quasisymmetric circulant matrices of order $d$ with entries from $\big\{0, \epsilon_1x_1, \ldots, \epsilon_kx_k\big\},$ where $\epsilon_{\ell}\in S_{\mathbb R},$ and the $x_{\ell}$'s are variables $(1\le \ell \le k)$, such that 
$$B_1B_1^*+\cdots +B_nB_n^*=\bigg(\sum_{\ell=1}^k u_{\ell}x_{\ell}^2\bigg)I_{d},$$
where the $u_{\ell}$'s are positive integers.
Then there is a circulant quasisymmetric $SOD\big(d; \  u_1, \ldots, u_k\big)$ for a signed group $S$ that admits a faithful remrep of degree $2^{n-1}$.
\end{theorem}
\begin{proof}\smartqed
Similar to the proof of Theorem \ref{induct}, but in here since $S_{\mathbb R}$ has the trivial remrep of degree 1, the final signed group $S$ will have a remrep of degree $2^{n-1}.$\qed
\end{proof}
\begin{example}\label{3tuplesod}
We explain how to use Theorem \ref{cortoreal} to find a $SOD\big(12; \ 4,4,4\big)$ for a signed group $S$ that admits a remrep of degree  
$8.$ Consider the following disjoint quasisymmetric circulant matrices of order $12$:

$B_1={\rm circ}\big(a,0,0,0,0,0,a,0,0,0,0,0\big),$

$B_2={\rm circ}\big(0,0,0,a,0,0,0,0,0,-a,0,0\big),$

$B_3={\rm circ}\big(0,b,c,0,0,0,0,0,0,0,c,-b\big),$

$B_4={\rm circ}\big(0,0,0,0,c,-b,0,-b,-c,0,0,0\big).$

Thus, $B_1B_1^*+B_2B_2^*+B_3B_3^*+B_4B_4^*=\big(4a^2+4b^2+4c^2\big)I_{12}.$
Apply Lemma \ref{ZS} to $B_1$ and $B_2$ to get a quasisymmetric normal circulant matrix of order $12$: $$A_1={\rm circ}\big(1a,0,0,\delta a,0,0,1a,0,0,-\delta a,0,0\big),$$ where  $\delta$ is in the signed group of order $2$: $$ S_1=\big<-1,\delta; \ \delta^2=1\big>$$ which admits a remrep of degree $2$ uniquely determined by
$1\rightarrow I_2$ and $\delta\rightarrow P.$ Since $B_1$ and $B_2$ are complementary, it follows that $A_1A_1^*=4a^2I_{12}.$ 

Applying Lemma \ref{ZS} again to $A_1$ and $B_3,$ there is a quasisymmetric normal circulant matrix of order $12$:
$$A_1={\rm circ}\big(1a,\gamma_1b,\gamma_2c,\gamma_3a,0,0,1a,0,0,-\gamma_3a,\gamma_2c,-\gamma_1b\big),$$
where $\gamma_1,\gamma_2,\gamma_3$ belong to the signed group of order $2^3$:
$$S_2=\Big<\gamma_1,\gamma_2,\gamma_3; \ \ \gamma_1^2=-\gamma_2^2=\gamma_3^2=1, \ \alpha\beta=-\beta\alpha; \ \alpha, \beta\in\{\gamma_1,\gamma_2,\gamma_3\}\Big>,$$ with a remrep of degree 4 which is uniquely determined by 
$$\gamma_1\rightarrow P\otimes I_2, \ \ \gamma_2\rightarrow R\otimes I_2, \ \ \gamma_3\rightarrow Q\otimes P.$$
Note that $A_2$ is not an SOD because $B_1, B_2$ and $B_3$ are not complementary.

Finally,  apply Lemma \ref{ZS} to $A_2$ and $B_4$ to get a quasisymmetric normal circulant matrix of order $12$:
$$A_3={\rm circ}\big(1a,\epsilon_1b,\epsilon_2c,\epsilon_3a,\epsilon_4c,-\epsilon_5b,1a,-\epsilon_5b,-\epsilon_4c,-\epsilon_3a,\epsilon_2c,-\epsilon_1b\big),$$
where $\epsilon_j, \ 1\le j\le 5$ belong to the signed group of order $2^5$:
$$S=\Big<\epsilon_1,\epsilon_2,\epsilon_3,\epsilon_4,\epsilon_5; \ \ \epsilon_1^2=-\epsilon_2^2=\epsilon_3^2=\epsilon_4^2=-\epsilon_5^2=1, \ \alpha\beta=-\beta\alpha; \ \alpha, \beta\in\{\epsilon_1,\epsilon_2,\epsilon_3,\epsilon_4,\epsilon_5\}\Big>,$$ with a remrep of degree 8 which is uniquely determined by 
$$\epsilon_1\rightarrow Q\otimes P\otimes I_2, \ \ \epsilon_2\rightarrow Q\otimes R\otimes I_2, \ \ \epsilon_3\rightarrow Q\otimes Q\otimes P, \ \ \epsilon_4\rightarrow P\otimes I_2\otimes I_2, \ \ \epsilon_5\rightarrow R\otimes I_2\otimes I_2.$$
So $A_3$ is a quasisymmetric circulant $SOD\big(12; \ 4,4,4\big).$ By Theorem \ref{sodtood}, there is an $$OD\big(8\cdot 12; \ 8\cdot 4,8 \cdot 4,8 \cdot 4\big).$$
\end{example}

Although Theorem \ref{cortoreal} shows that the degree of remrep is $2$ times less than the one in Theorem \ref{induct},  we have more complex Golay pairs than real ones. Thus, 
from now on, we just consider the complex case, and we refer the reader to \cite[chap. 6]{EGh} for the results that obtain from the real case. 

For $u$ a positive integer, denote by $\ell c(u)$ the least number of  complex Golay numbers that add up to $u$, and let $\ell c(0)=0$. 
Also, denote by $\ell' c(u)$ the least number of  complex Golay numbers in two variables that add up to $u$.
Indeed, $\ell' c(2u)\le \ell c(u).$ Note that Lemma \ref{twovariables} insures the existence of a complex Golay pair in two variables of length $2m$ if there exists a complex Golay pair of length $m$.

In the following theorem, we show how to use complex Golay pair and complex Golay pairs in two variables to construct SODs. 
\begin{theorem}\label{KTUPLESOD}
Let $\big(1, v_1, \ldots, v_{q}, w_1, w_1, \ldots, w_t, w_t\big)$ be a sequence of positive integers such that $v_i$'s, $1\le i\le q$, are disjoint and let $1+\sum_{\beta=1}^q v_{\beta}+2\sum_{\delta=1}^tw_{\delta}=u.$ Then there exists a full circulant quasisymmetric $SOD\big(4u; \ 4, 4v_1, \ldots, 4v_q, 4w_1, 4w_1, \ldots ,4w_t,4w_t\big)$ for some signed group $S$ that admits a remrep of degree $2^{n},$ where $n\le 2+2\sum_{\beta=1}^q  \ell c(v_{\beta})+2\sum_{\delta=1}^t\ell c(w_{\delta}).$ 
\end{theorem}
\begin{proof}\smartqed
For each $\beta$, $1\le \beta \le q$, and each $\alpha$, $1\le \alpha \le \ell c(v_{\beta})$, let $\big(A[\alpha, v_{\beta}]; \ B[\alpha, v_{\beta}]\big)$ be a complex Golay pair in one variable $x_{\beta}$ of length $V[\alpha, v_{\beta}].$ From the definition of $\ell c(v_{\beta})$, for each $\beta$, $1\le \beta \le q,$  $\sum_{\alpha=1}^{\ell c(v_{\beta})}V[\alpha, v_{\beta}]=v_{\beta}.$ Let $S[\alpha, \beta]:=\sum_{i=1}^{\alpha-1}V[i,v_{\beta}]+\sum_{j=1}^{\beta-1}v_j.$
Also, for each $\delta$, $1\le \delta \le t$, and each $\gamma$, $1\le \gamma \le \ell' c(2w_{\delta})$, let $\big(C[\gamma, w_{\delta}]; \ D[\gamma, w_{\delta}]\big)$ be a
complex Golay pair of length  $W[\gamma, w_{\delta}]$ in two variables $y_{\delta}$ and $z_{\delta}$.
From the definition of $\ell' c(2w_{\delta})$,  for each $\delta$, $1\le \delta \le t,$  $\sum_{\gamma=1}^{\ell' c(2w_{\delta})}W[\gamma, w_{\delta}]=2w_{\delta}.$ Let $S'[\gamma, \delta]:=\sum_{i=1}^{\gamma-1}W[i,w_{\delta}]+2\sum_{j=1}^{\delta-1}w_j.$
For each $\beta$, $1\le \beta \le q$, and each $\alpha$, $1\le \alpha \le \ell c(v_{\beta})$, and for each $\delta$, $1\le \delta \le t$ and each $\gamma$, $1\le \gamma \le \ell' c(2w_{\delta})$, the following are $ n=2+2\sum_{\beta=1}^q  \ell c(v_{\beta})+2\sum_{\delta=1}^t\ell' c(2w_{\delta})$ circulant matrices of order $4u$:
\begin{align*}
&M_1={\rm circ}\big(x, \ {0}_{(2u-1)}, \ x, \ { 0}_{(2u-1)}\big), \\
&M_2={\rm circ}\big({ 0}_{(u)}, \ -x,\  { 0}_{(2u-1)}, \ x,\ { 0}_{(u-1)}\big), \\
 &X_{\alpha \beta}={\rm circ}\big( { 0}_{(S[\alpha, \beta]+1)}, \ A[\alpha, v_{\beta}], \ { 0}_{(4u-2S[\alpha+1, \beta]-1)}, \ B[\alpha, v_{\beta}],\ { 0}_{(S[\alpha, \beta])}\big), \\
&Y_{\alpha \beta}={\rm circ}\big({ 0}_{(2u-S[\alpha+1, \beta])},\  -B[\alpha, v_{\beta}], \ { 0}_{(2S[\alpha, \beta]+1)},\ A[\alpha, v_{\beta}], \ { 0}_{(2u-S[\alpha+1, \beta]-1)}\big), \\
&Z_{\gamma \delta}={\rm circ}\big( { 0}_{(v+S'[\gamma, \delta]+1)}, \ C[\gamma, w_{\delta}], \ { 0}_{(4u-2v-2S'[\gamma+1, \delta]-1)}, \ D[\gamma, w_{\delta}],\ { 0}_{(v+S'[\gamma, \delta])}\big), \\
&T_{\gamma \delta}={\rm circ}\big({ 0}_{(2u-v-S'[\gamma+1, \delta])},\  -D[\gamma, w_{\delta}], \ { 0}_{(2S'[\gamma, \delta]+2v+1)},\ C[\gamma, w_{\delta}], \ { 0}_{(2u-v-S'[\gamma+1, \delta]-1)}\big).
\end{align*}
It can be seen that the above circulant matrices are disjoint and quasisymmetric such that 
\begin{align*}
\sum_{i=1}^2 M_iM_i^*&+\sum_{\beta=1}^q\sum_{\alpha=1}^{\ell c(v_{\beta})}\big(X_{\alpha \beta}X^*_{\alpha \beta}+Y_{\alpha \beta}Y^*_{\alpha \beta}\big)+\sum_{\delta=1}^t\sum_{\gamma=1}^{\ell' c(2w_{\delta})}\big(Z_{\gamma \delta}Z^*_{\gamma \delta}+T_{\gamma \delta}T^*_{\gamma \delta}\big)\\ &=4\bigg(x^2+\sum_{\beta=1}^q\big(v_{\beta}x_{\beta}^2\big)+\sum_{\delta=1}^t\big(w_{\delta}y_{\delta}^2+w_{\delta}z_{\delta}^2\big)\bigg)I_{4u}.
\end{align*}
Thus, by Theorem \ref{induct}, there exists a full circulant quasisymmetric
$$SOD\big(4u; \ 4, 4v_1, \ldots, 4v_q, 4w_1, 4w_1, \ldots ,4w_t,4w_t\big)$$ 
for a signed group $S$ which admits a remrep of degree $2^{n}$, where 
\begin{align*}
\displaystyle n=2+2\sum_{\beta=1}^q  \ell c(v_{\beta})+2\sum_{\delta=1}^t\ell' c(2w_{\delta})\le 2+2\sum_{\beta=1}^q  \ell c(v_{\beta})+2\sum_{\delta=1}^t\ell c(w_{\delta}).
\end{align*}\qed
\end{proof}
\begin{example}
Consider the $4$-tuple $(1,v_1,v_2,v_3)=(1,5,7,17)$.  By Theorem \ref{KTUPLESOD}, there is a circulant quasisymmetric
$SOD\big(4\cdot 30; \ 4\cdot 1, 4\cdot 5, 4\cdot 7, 4\cdot 17\big),$
which admits a remrep of degree $2^{n}$, where $n=2+2\ell c(5)+2\ell c(7)+2\ell c(17)=2+2+4+4=12.$ By Theorem \ref{sodtood},
there is an 
$$OD\big(2^{14}\cdot 30; \ 2^{14}\cdot 1, 2^{14}\cdot 5, 2^{14}\cdot 7, 2^{14}\cdot 17\big).$$
\end{example}
\begin{example}
Let $(1,w_1,w_1,w_2,w_2,w_3,w_3,w_4,w_4)=(1,3,3,5,5,11,11,13,13)$.  By Theorem \ref{KTUPLESOD}, there is a circulant quasisymmetric
$$SOD\big(4\cdot 65; \ 4\cdot 1, 4\cdot 3_{(2)}, 4\cdot 5_{(2)}, 4\cdot 11_{(2)},4\cdot 13_{(2)}\big),$$
which admits a remrep of degree $2^{n}$, where $n=2+2\ell c(3)+2\ell c(5)+2\ell c(11)+2\ell c(13)=10.$ By Theorem \ref{sodtood},
there is an $$OD\Big(2^{12}\cdot 65; \ 2^{12}\cdot 1, 2^{12}\cdot 3_{(2)}, 2^{12}\cdot 5_{(2)}, 2^{12}\cdot 11_{(2)}, 2^{12}\cdot 13_{(2)}\Big).$$
\end{example}
\section{Bounds for the asymptotic existence orthogonal designs}\label{sec4}

In this section, we obtain some upper bounds for the degree of remrep in Theorem \ref{KTUPLESOD}, and then we find some upper bounds for the asymptotic existence of ODs.

To get a better upper bound for the degree of remrep for any $k$-tuple $\big(u_1, u_2, \ldots, u_k\big)$ of positive integers, from now on,
 we assume that  $\ell c(u_1)-\ell c(u_1-1)$ is greater than or equal to $\ell c(u_i)-\ell c(u_i-1)$ for all $2\le i\le k$. We also define $\log(0)=0$, and in here the base of $\log$ is $2$.


Livinskyi \cite[chap. 5]{Ivan}, by a computer search,  showed that each positive integer $u$ can be presented as sum of at most  $3\big\lfloor \log_{2^{26}} (u) \big\rfloor+4$ complex Golay numbers. Thus
\begin{align}\label{13bound}
\ell c(u)\le 3\Big\lfloor \dfrac{1}{26}\log (u) \Big\rfloor+4\le \dfrac{3}{26}\log (u)+4.
\end{align}
 \begin{theorem}\label{better}
Suppose that  $\big(u_1, u_2, \ldots, u_k\big)$ is a $k$-tuple of positive integers and let $u_1+\cdots +u_k=u.$ 
Then there is a full circulant quasisymmetric $SOD\big(4u; \ 4u_1, 4u_2, \ldots, 4u_k \big)$ for some signed group $S$ that admits a remrep of degree $2^{n},$ where $ n\le (3/13)\log (u_1-1)+(3/13)\sum_{i=2}^k \log (u_i)+8k+2 .$ 
\end{theorem}
\begin{proof}\smartqed
Apply Theorem \ref{KTUPLESOD} to the $(k+1)$-tuple $\big(1,u_1-1, u_2,\ldots, u_k\big)$. So  there is a full circulant quasisymmetric 
$SOD\big(4u; \ 4u_1, 4u_2, \ldots, 4u_k \big)$ for some signed group $S$ that admits a remrep of degree $2^{n},$ where $ n\le 2+2\ell c(u_1-1)+2\sum_{i=2 }^k \ell c(u_i).$ Use  \eqref{13bound} to obtain the desired.\qed
\end{proof}
\begin{remark}\label{secondbound}
For any given $k$-tuple $\big(u_1, u_2, \ldots, u_k\big)$ of positive integers, one may write it as the $(k+1)$-tuple $\big(1, u_1-1, u_2, \ldots, u_k\big)$, and then sort its elements to get the $(k+1)$-tuple $\big(1,v_1, \ldots , v_q, w_1, w_1, \ldots, w_t, w_t\big),$ where $v_i$'s are disjoint and then use Theorem \ref{KTUPLESOD} and   \eqref{13bound} to obtain the following bound:
\begin{align*}
n&\le 2+2\sum_{i=1}^q  \ell c(v_{i})+2\sum_{j=1}^t\ell c(w_{j})\\ &\le 
2+\dfrac{3}{13}\sum_{i=1}^q\log (v_i)+\dfrac{3}{13}\sum_{j=1}^t \log (w_j)+8(q+t),
\end{align*}
where $n$ is the exponent of the degree of remrep.
\end{remark}

By Theorem \ref{sodtood} and Theorem \ref{better}, we have the following asymptotic existence result.
\begin{theorem}\label{thirdbound}
Suppose $\big(u_1, u_2, \ldots, u_k\big)$ is a $k$-tuple of positive integers. Then for each 
$n\ge N$, there is an $OD\big(2^n\sum_{j=1}^ku_j; \ 2^nu_1, \ldots, 2^nu_k\big),$
where $N\le (3/13)\log (u_1-1)+(3/13)\sum_{i=2}^k \log (u_i)+8k+4.$
\end{theorem}

Livinskyi \cite[chap. 5]{Ivan} used complex Golay, Base, Normal and other sequences  (see \cite{Dokovic,Kouk2, Koun2,Koun}) to show that  each positive integer $u$ can be presented as sum of 
\begin{align} \label{30bound}
s\le \dfrac{1}{10}\log (u) +5 
\end{align}
pairs $(A_k[u];B_k[u])$ for $1\le k\le s$ such that $A_k[u]$ and $B_k[u]$ have the same length for each $k$, $1\le k\le s$, with elements from $\{\pm 1, \pm i\}$, and the set 
$\big\{A_1[u], B_1[u], \ldots , A_s[u], B_s[u]\big\}$ is a set of complex complementary sequences with weight $2u$.  In the following theorem, we use this set of complex complementary sequences.
\begin{theorem}\label{lastbound}
Suppose $\big(v_1, v_2, \ldots, v_k\big)$ is a $k$-tuple of positive integers. Then for each 
$n\ge N$, there is an $OD\big(2^n\sum_{j=1}^kv_j; \ 2^nv_1, \ldots, 2^nv_k\big),$
where $N\le (1/5)\log (v_1-1)+(1/5)\sum_{i=2}^k \log (v_i)+10k+4.$
\end{theorem}
\begin{proof}\smartqed
Suppose $\big(v_1, v_2, \ldots, v_k\big)$ is a $k$-tuple of positive integer. Let $\sum_{j=1}^kv_j=v.$
For simplicity, we assume that $u_1=v_1-1$ and $u_i=v_i$ for $2\le i\le k$.

For each $\beta$, $1\le \beta\le  k$, let $\big\{A_1[u_{\beta}], B_1[u_{\beta}], \ldots , A_{s_{\beta}}[u_{\beta}], B_{s_{\beta}}[u_{\beta}]\big\}$ be a set of complex complementary sequences with weight $2u_{\beta}$ such that for each $\alpha$, $1\le \alpha \le s_{\beta}$,
$A_{\alpha}[u_{\beta}]$ and $B_{\alpha}[u_{\beta}]$ have the same length, $V[\alpha, u_{\beta}]$. From \eqref{30bound}, for each $\beta$, $1\le \beta\le  k$,
\begin{align}\label{sbeta}
s_{\beta}\le \dfrac{1}{10}\log u_{\beta}+5.
\end{align}
Suppose that $x$ and $x_{\beta}$, $1\le \beta\le k$ are variables. 
Let $M_1={\rm circ}\big(x, \ { 0}_{(2v-1)}, \ x, \ { 0}_{(2v-1)}\big)$ and $M_2={\rm circ}\big({ 0}_{(v)}, \ -x,\  { 0}_{(2v-1)}, \ x,\ { 0}_{(v-1)}\big).$
For each $\beta$, $1\le \beta\le  k$, and  each $\alpha$, $1\le \alpha \le s_{\beta}$, let
\begin{align*}
&X_{\alpha \beta}={\rm circ}\Big( 0_{(S[\alpha, \beta]+1)}, \ x_{\beta}A_{\alpha}[u_{\beta}], \ 0_{(4v-2S[\alpha+1, \beta]-1)}, \ x_{\beta}B_{\alpha}[u_{\beta}],\ 0_{(S[\alpha, \beta])}\Big), \\
&Y_{\alpha \beta}={\rm circ}\Big(0_{(2v-S[\alpha+1, \beta])},\  -x_{\beta}B_{\alpha}[u_{\beta}], \ 0_{(2S[\alpha, \beta]+1)},\ x_{\beta}A_{\alpha}[u_{\beta}], \ 0_{(2v-S[\alpha+1, \beta]-1)}\Big),
\end{align*}
where $S[1,1]=0$ and $ S[a, b]=\sum_{j=1}^{a-1} \sum_{i=1}^{b} V[j,u_i]$, for $1\le b\le k$ and $1<a\le s_{b}+1$.
It can be seen that the above circulant matrices are disjoint and quasisymmetric of order $4v$ such that 
\begin{align*}
\sum_{i=1}^2 M_iM_i^*+\sum_{\beta=1}^k\sum_{\alpha=1}^{s_{\beta}}\big(X_{\alpha \beta}X^*_{\alpha \beta}+Y_{\alpha \beta}Y^*_{\alpha \beta}\big) =4\Big(x^2+\sum_{\beta=1}^k\big(u_{\beta}x_{\beta}^2\big)\Big)I_{4v}.
\end{align*}
Thus, by Theorem \ref{induct}, there exists a full circulant quasisymmetric
$SOD\big(4v; \ 4, 4u_1, \ldots, 4u_k\big)$ 
for a signed group $S$ which admits a remrep of degree $2^{m}$, where $ m=2+2\sum_{\beta=1}^k s_{\beta}.$ 
From Theorem \ref{sodtood} and the upper bounds for the $s_{\beta}$'s, \eqref{sbeta}, there is an $OD\big(2^nv; \ 2^n, 2^nu_1, \ldots, 2^nu_k\big),$ and so there is
an $OD\big(2^nv; \ 2^nv_1, \ldots, 2^nv_k\big),$ where $n\le (1/5)\log (v_1-1)+(1/5)\sum_{i=2}^k \log (v_i)+10k+4$.\qed
\end{proof}
\begin{acknowledgement}
The paper constitutes a part of the author's Ph.D. thesis written under the direction of Professor Hadi Kharaghani at the University of Lethbridge. 
The author would like to thank Professor Hadi Kharaghani for introducing the problem and his very useful guidance toward solving the problem and also Professor Rob Craigen for his time and great help.  
\end{acknowledgement}

\bibliography{Ghaderpour}{}
\bibliographystyle{spmpsci}

\end{document}